\documentclass{article}
\usepackage[utf8]{inputenc}
\usepackage[utf8]{inputenc}
\usepackage{geometry}
\usepackage{amsmath}
\usepackage{hyperref}
\usepackage{amsthm}
\usepackage{amssymb}
\usepackage{xcolor}
\usepackage{comment}
\usepackage{tensor}
\usepackage{graphicx}
\usepackage{caption}
\usepackage{subcaption}
\numberwithin{equation}{section}
\newtheorem{theorem}{Theorem}[section]
\newtheorem*{theorem*}{Theorem}
\newtheorem{proposition}{Proposition}[section]
\newtheorem{remark}{Remark}[section]
\newtheorem{lemma}{Lemma}[section]

\newtheorem{corollary}{Corollary}[section]
\newtheorem{porism}{Porism}[section]
\title{Joint evolution of a\\ Lorentz-covariant massless scalar field\\ and its point-charge source\\ in one space dimension}
\author{Lawrence Frolov$^*$, Samuel Leigh$^\dagger$, and Shadi Tahvildar-Zadeh$^*$\\[10pt]
{\small $^*$ Department of Mathematics, Rutgers University (New Brunswick)}\\
{\small $^\dagger$ Department of Physics, University of Washington}
}
\date{May 2024}

\begin{document}

\maketitle
\begin{abstract}
    In this paper we prove that the static solution of the Cauchy problem for a massless real scalar field that is sourced by a point charge in $1+1$ dimensions is asymptotically stable under perturbation by compactly-supported radiation.  This behavior is due to the process of back-reaction. Taking the approach of Kiessling, we rigorously derive the expression for the force on the  particle from the principle of total energy-momentum conservation. We provide a simple, closed form  for the   particle's self-action, and show that it is restorative   in this model , i.e. proportional to negative velocity, and causes the charge to return to rest after the radiation passes through. We establish these results by studying the joint evolution problem for the particle-scalar field system, and proving its global well-posedness and the claimed asymptotic behavior.
\end{abstract}
\newpage
\section{Introduction and Statement of Main Results}

\subsection{Background} Consider the dynamics of a vibrating point charge and the electromagnetic field it is sourcing. As the particle oscillates it radiates electromagnetic waves which propagate away from the charge at the speed of light. These waves carry both energy and momentum, so to conserve the total energy-momentum of the system, the particle must undergo some dampening through an interaction with its own field. This dampening self-interaction is one example of back-reaction, the process in which charge distributions source fields which then ``re-act" on the charges. 

As it is well-known, attempting to study the process of back-reaction in the framework of Maxwell-Lorentz electrodynamics leads to a fundamental inconsistency. The Lorentz force needed to calculate the path of the particle requires us to evaluate the electromagnetic field along the particle's path, and yet the field sourced by the particle is undefined precisely on this path! Resolving this inconsistency has been an open problem for more than a century, and has been worked on by many notable figures such as Abraham \cite{Abraham}, Poincaré \cite{Poincare} and Dirac \cite{Dirac}. A gripping account of this endeavor can be found in \cite{spohn_2004}, which also includes an excellent review of the main approach mathematicians have taken to successfully resolve the inconsistency, namely smearing the point charge into a smooth charge distribution.  In this approach however, it is not possible to take the smearing away, once it is introduced. The introduction of smearing also complicates the task of keeping the system of equations fully Lorentz-covariant, leading to many of the results obtained thus far being restricted to non-relativistic motions of the particle. (See also \cite{Ko20}, Section 2.3, for more recent results in this direction.)

Previous techniques of studying these field-particle systems directly, either without smearing or with smearing that is put in and then taken away, has left something to be desired. Outlined in a recent review \cite{Poisson}, they include an infinite bare-mass renormalization that is mathematically ill-defined, or an ad-hoc averaging of the fields in a neighborhood of the charge. A breakthrough occurred in 2019 when, following up on the work of Poincaré \cite{Poincare}, Kiessling \cite{Kiessling, KiesslingE} showed that postulating energy-momentum conservation of the field-particle system yields a unique and admissible force law, provided that the field's momentum density is locally integrable around the particle. Although this integrability assumption rules out the classical vacuum law of Maxwell, given by $E=D$, $B=H$, it admits others, such as the Bopp-Landé-Thomas-Podolsky (BLTP) vacuum law \cite{Bopp1, Bopp2,Lande,Thomas,Podolsky}.

In three dimensions, the Maxwell electromagnetic field sourced by a point charge is not only too singular to evaluate along the charge's path, but also its energy and momentum densities are not locally integrable around the source. By contrast BLTP, which is a higher order modification of Maxwell, does have the regularity necessary to derive a unique force law from the assumption of energy-momentum conservation. In the context of BLTP, Kiessling and Tahvildar-Zadeh \cite{KTZ} have successfully applied this force law to prove local well-posedness of the joint field-particle dynamics, and Hoang {\em et al.} \cite{Hoang} proved global existence for the scattering problem of a single charge interacting with a smooth potential. However, the complex form of the BLTP energy-momentum conserving force law has so far resisted a clear analysis of the asymptotic behavior of the particle. 

This paper studies the dynamics of a relativistic point charge coupled to a massless scalar field on flat $1+1$ dimensional space-time. Working in one space dimension provides us with the regularity needed to derive the energy-momentum conserving force law without making any higher order modifications to the theory, thus allowing a much simpler analysis to be performed. We choose to study scalar charges because electromagnetism is not viable in one space dimension. Our model closely resembles the one studied in \cite{SEA}, with the exception that our dynamics are fully relativistic. To isolate the effects of the interaction between the scalar charge and its own field, we will be focusing on the case of a single particle perturbed by
scalar radiation.
\subsection{Main Results}

Taking Kiessling's approach, we show that the 2-force which acts on a single charged particle at $z=(z^0,z^1)\in \mathbb{R}^{1,1}$ is given by
\begin{equation}\label{Intro force law}
    F^\mu(z^0,z^1)=-\bigg[n_\nu T^{\mu \nu}_S (z^0,x^1)\bigg]_{x^1=z^1},
\end{equation}
where $n_\mu$ is the unit covector that is annihilated by the particle's two-velocity, $T_S^{\mu \nu}$ is the energy-momentum tensor of the field, and $[\cdot]_{x^1=z^1}$ denotes the jump in space at $z^1$. The force law given by equation (\ref{Intro force law}) is derived from the principle of energy-momentum conservation
\begin{equation}
    \partial_\nu T_p^{\mu \nu}+\partial_\nu T_S^{\mu \nu}=0,
\end{equation}
where $T_p^{\mu \nu}$ is the energy-momentum tensor of the particle, which is concentrated on the particle's world-line, and the derivatives are taken weakly, to account for singularities in the two energy tensors. %In words, the force acting on a point charge is opposite the energy-momentum flux of the field localized around the particle.

\par
Guided by Weyl's ``{\em agens} theory" of matter \cite{Weyl}, according to which the world-lines of matter particles are simply the locus of singularities of the underlying spacetime and/or the fields defined on that spacetime, we are led to the study of  a joint evolution problem in which the path of the particle $z(\tau)$ appears as a jump discontinuity in the derivatives of the scalar field $U$, with the jumps showing up inside the force term in the equation for $\ddot{z}$. The equation of motion for a scalar field with a point charge source is
\begin{equation}\label{Intro field equation}
    \eta^{\mu \nu}\partial_\mu \partial_\nu U(x)=a \int \delta^{(2)}(x-z(\tau))d\tau 
\end{equation}
where $\tau$ and $a$ are the particle's proper time and scalar charge respectively. We consider a stationary field-particle system which is perturbed by some incoming scalar ``radiation." The joint evolution problem corresponding to this is given by the initial value problem for the unknown field $U = U(x)$ and unknown trajectory $z = (z^\mu(\tau))$ satisfying
\begin{equation}\label{Intro Scalar IVP}
\left\{\begin{array}{rcl}
\eta^{\mu \nu}\partial_\mu \partial_\nu U(x) &=& a\int \delta^{(2)}(x-z(\tau))d\tau 
\\
U(0,x^1)&=&-\frac{a}{2}|x^1|+V_0(x^1)
\\
\partial_0 U(0,x^1)&=&V_1(x^1),
\end{array}\right.
\end{equation}
\begin{equation}\label{Intro Charge IVP}
\begin{cases}
\frac{dz^\mu}{d\tau}(\tau)=:u^\mu(\tau)
\\
\frac{dp^\mu}{d\tau}(\tau)=F^\mu(z(\tau)),\qquad p^\mu(\tau) :=(\tilde m-aU(z(\tau)))u^\mu(\tau).
\end{cases}
\end{equation}
where $\tilde m$ is the bare mass of the particle and $F$ is as in \eqref{Intro force law}.  We work in the fixed Lorentz frame where the particle remained stationary at the origin for all time $x^0<0$. Thus the data for \eqref{Intro Charge IVP} is $z(\tau) = (0,0)$ and $p(\tau) = (\tilde m, 0)$ for all $\tau\leq0$. The motion of the charge is perturbed by incoming radiation, which is represented by $V_0$, $V_1$ in the initial data for $U$.
The following is an informal statement of the first main result for this paper (for the precise statement, see Theorem \ref{Global}):
\begin{theorem*}\label{Intro Global}
For any set of particle parameters with positive bare mass and non-zero real scalar charge, and for any set of small, smooth, compactly supported  functions $V_0(x^1)$, $V_1(x^1)$, the joint initial value problem given by (\ref{Intro Scalar IVP}) and (\ref{Intro Charge IVP}) admits a unique, global-in-time solution.
\end{theorem*}
To prove this, we explicitly compute the forces in equation (\ref{Intro force law}), and from that we extract a closed form for the force that scalar point charges exert   on themselves.
We show that the scalar self-force resulting from this back-reaction is proportional  to the negative of  the particle's velocity, see equation (\ref{Joint 1}). 
 
\begin{remark}\label{Cancellation}
    We were able to determine this self-force because scalar fields sourced by point charges are regular enough in one space dimension for radiation to be emitted at a finite rate. However, the interaction between scalar fields and their charges generate mass (see (\ref{dyn mass})), and we show  that there is cancellation between said self-force and the particle's loss in mass  (see (\ref{product rule})).  Hence, there is a sense in which the {\em radiation}-reaction in this model does not contribute to the proper acceleration of the point charge, only the {\em back}-reaction on the particle of the static field in its own past does.  
\end{remark}
 
The following is an informal statement of the second main result for this paper (for the precise statement, see Theorem \ref{Rest}):
\begin{theorem*}
Let $z$, $U$ satisfy the joint IVP given by (\ref{Intro Scalar IVP}--\ref{Intro Charge IVP}) for a single scalar particle with positive bare mass and non-zero real scalar charge, and $V_0$, $V_1$ small, smooth, and compactly supported. Then $\lim_{x^0 \to \infty}u^1(x^0)=0$.
\end{theorem*}
 Accordingly (by Lorentz covariance), a charged particle undergoing uniform motion in some frame, when perturbed by some compactly supported radiation, asymptotically returns to that uniform motion.
This stability is not a result of radiation-reaction, for reasons explained in Remark~\ref{Cancellation}. It is instead a consequence of the charge's interaction with its own freely evolving field coming from the $-\frac{a}{2} |x|$ term in the initial data of (\ref{Intro Scalar IVP}). 
\begin{remark}
    In forthcoming work \cite{FETZ}, we show that such an asymptotic stability result does not hold for charges coupled to {\em massive} scalar fields in one space dimension. This is because the stationary particle's massive field has finite energy, as a result of which the charge's interaction with the field decays sufficiently rapidly in time, so that the motion may asymptote to a different steady state, which is a form of {\em orbital} stability. Asymptotic stability may thus be specific to charges coupled to massless scalar fields in one dimension. Indeed, if the dynamics of smeared-out charges \cite{spohn_2004} can be a guide, it is likely that this will {\em not} hold in general for electromagnetic point charges, or in higher dimensions.
\end{remark}

\par
In section 2 we derive the equations of motion for a scalar field from the principle of stationary action, and derive the force law from the assumption of energy conservation. In section 3 we present a proof of the global well-posedness result for the joint evolution of a single scalar particle and its field. In section 4 we study the asymptotics of our joint evolution, and provide a proof of the asymptotic stability of the stationary solution.
\section{Derivation of Equations of Motion}
\subsection{Field Equations From Principle of Stationary Action}
We derive the equations of motion for our scalar field using the principle of stationary action. The action for a point charge coupled to a massless scalar field $U$ defined on $1+1$ dimensional flat space-time\footnote{We work with signature $\eta_{\mu \nu}=\text{diag}(1,-1)$.} is given by
\begin{equation}\label{action}
    S[U,z]=\int\mathcal{L} \sqrt{-\eta} d^2x,
\end{equation}
 where the Lagrangian density $\mathcal{L}$ is defined via
\begin{equation}
\begin{split}
    \mathcal{L}(x):=&\frac{1}{\sqrt{-\eta}}  \int  -(\tilde m-aU(z))\sqrt{\eta_{\mu \nu}\dot{z}^\mu \dot{z}^\nu}\delta^{(2)}(x-z)d\theta 
    \\
   & + \frac{1}{2}\eta^{\mu \nu}\partial_\mu U \partial_\nu U(x).
\end{split}
\end{equation}
Here $z^\mu$, $\tilde m$, and $a$ represent the particle's space-time position, bare mass, and scalar charge respectively, while $\theta$ is an arbitrary parameterization of the particle's worldline.
Extremizing the action by taking variations with respect to $U$ returns 
\begin{equation}\label{field equation}
    \eta^{\mu \nu}\partial_\mu \partial_\nu U= a \int \sqrt{\eta_{\mu \nu}\dot{z}^\mu \dot{z}^\nu}\delta^{(2)}(x-z)d\theta= a\int \delta^{(2)}(x-z(\tau))d\tau, 
\end{equation}
where $\tau$ is the particle's proper time defined by $d\tau=\sqrt{\eta_{\mu \nu}\dot{z}^\mu \dot{z}^\nu}d\theta$. Equation (\ref{field equation}) shows that the point charge is acting as a singularity in the derivatives of the scalar field $U$. Given a world-line for our particle along with some specified initial data for $U$, we could solve for the evolution of $U$ by solving the associated initial value problem. However, we are interested in studying the joint evolution problem of our field-particle system. So the motion of the charge must be in accordance with all the scalar forces that act on it. We may naively derive the forces acting on our charges by taking variations of the action with respect to $z^\mu$. This leads to a familiar law of motion that is inconsistent with the field equations, namely:
\begin{equation}\label{Lorentz force}
    \frac{d p^\mu}{d \tau}=-a\partial^\mu U(z),
\end{equation}
where $p$ is the {\em dynamical momentum} defined by
\begin{equation}\label{dyn momentum}
    p^\mu :=(\tilde m-aU(z))u^\mu, \quad u^\mu:=\frac{dz^\mu}{d\tau}.
\end{equation}

The inconsistency arises since the force law \eqref{Lorentz force} derived from the principle of stationary action requires one to evaluate the first derivatives of the scalar field along the particle's world-line. However, the field equations derived from the same principle imply that the first derivatives of the field are undefined precisely along this world-line. So there can be no joint evolution of the field-particle system which extremizes the action. 

In non-rigorous settings, one is typically taught to ignore the ill-defined self-interactions terms. But the dynamics given by ignoring singular self-forces is off-shell. In particular, it does not conserve the total energy-momentum of the system, as we will show later. The problems we face here are not unique to scalar point particles, and we take inspiration from the work of Kiessling on force laws for electromagnetic particles in three space-dimensions \cite{Kiessling}, to derive a rigorous force law for our scalar point charges. 
\par 
We conclude this section by defining the {\em dynamical mass} $m$ which depends on the field $U$ evaluated along the particle's world-line
\begin{equation}\label{dyn mass}
    m(x^0):=\frac{p^\mu}{u^\mu}=\tilde m-aU(x^0,z^1(x^0)),
\end{equation}
where 
\begin{equation}
    z^1(x^0):=z^1(\tau^{-1}(x^0)),
\end{equation}
and $\tau^{-1}(x^0)$ is the unique value for which $z^0(\tau^{-1}(x^0))=x^0$.
\begin{remark}
The dynamical mass and momentum will frequently appear throughout our calculations, taking the place of mass and momentum wherever one would expect the latter two to appear.
\end{remark}
\subsection{Force Law From Conservation of Energy-Momentum}
In this section we will derive our force law from the assumption of the local conservation of energy-momentum. 
%\footnote{This assumption should not be considered as taken in addition to the principle of stationary action. If there were joint-evolutions which extremized the action, they would satisfy the same law of energy-momentum conservation.}.  
From the action \eqref{action} we derive the following energy density-momentum density-stress tensor ({\em energy tensor}, for short) for our system:
\begin{equation}
\begin{split}
    T^{\mu \nu}(x)= T_{p}^{\mu \nu}(x)+T_{S}^{\mu \nu}(x),
\end{split}
\end{equation}
where 
\begin{equation}
    T_p^{\mu \nu}:=m(x^0)\int u^\mu u^\nu \delta^{(2)}(x-z(\tau))d\tau, \hspace{0.5cm} T_S^{\mu \nu}:=(\partial^\mu U)(\partial^\nu U)-\frac{1}{2}\eta^{\mu \nu}(\partial_\alpha U)(\partial^\alpha U).
\end{equation}
The total energy tensor is singular along the particle's world-line, so we will take conservation of energy-momentum in the weak sense of distributions to derive our force law.
\begin{theorem}
Suppose $U$ is such that
\begin{equation}
   \forall x^0\geq 0, \ \lim_{\epsilon \rightarrow 0}\int_{z^1(x^0)-\epsilon}^{z^1(x^0)+\epsilon}T^{0\nu}_S(x^0,x^1) dx^1=0.
\end{equation}
(This requires $T_S^{0 \nu}$ be locally integrable\footnote{While this condition does not hold for fields sourced by scalar point charges in three space dimensions, it holds in one space dimension.}  around $z$.) Then, assuming  energy-momentum conservation in the weak sense of 
\begin{equation}
    \int_{\Omega}\partial_\mu T^{\mu 0}dV=0=\int_{\Omega}\partial_\mu T^{\mu 1}dV,
\end{equation}
for all tubular regions $\Omega$ around the particle's world-line, yields the unique force law
\begin{equation}\label{K Force 1}
    \frac{d p^{\nu}}{d \tau}=-\bigg[n_\mu T^{\mu \nu}_S(z^0,x^1)\bigg]_{x^1=z^1},
\end{equation}
where $n_\mu=(-u^1,u^0)$ is the space-like unit covector annihilated by $u^\mu$.
\end{theorem}
\begin{proof}
Fix $\nu$. By the definition of weak derivatives,
\begin{equation}
\int_\Omega \partial_\mu T^{\mu \nu} dV=\int_{\partial \Omega} T^{\alpha \nu} N_\alpha dS,
\end{equation}
where $N$ are the unit normal vectors to our boundary $\partial \Omega$ and $dS$ is the surface element induced by $\eta$. The boundary $\partial \Omega$ consists of four parts (see Figure~\ref{fig:region}): Two space-like curves given by 
\begin{equation}
\begin{split}
    &\bar{T}_1=\{x\in \mathbb{R}^{1,1}| z^1(T_1)-\epsilon<x^1<z^1(T_1)+\epsilon, x^0=T_1 \}
    \\
    &\bar{T}_2=\{x\in \mathbb{R}^{1,1}| z^1(T_2)-\epsilon<x^1<z^1(T_2)+\epsilon, x^0=T_2 \}
\end{split}
\end{equation}
\begin{figure}
\centering
\includegraphics[scale=0.6]{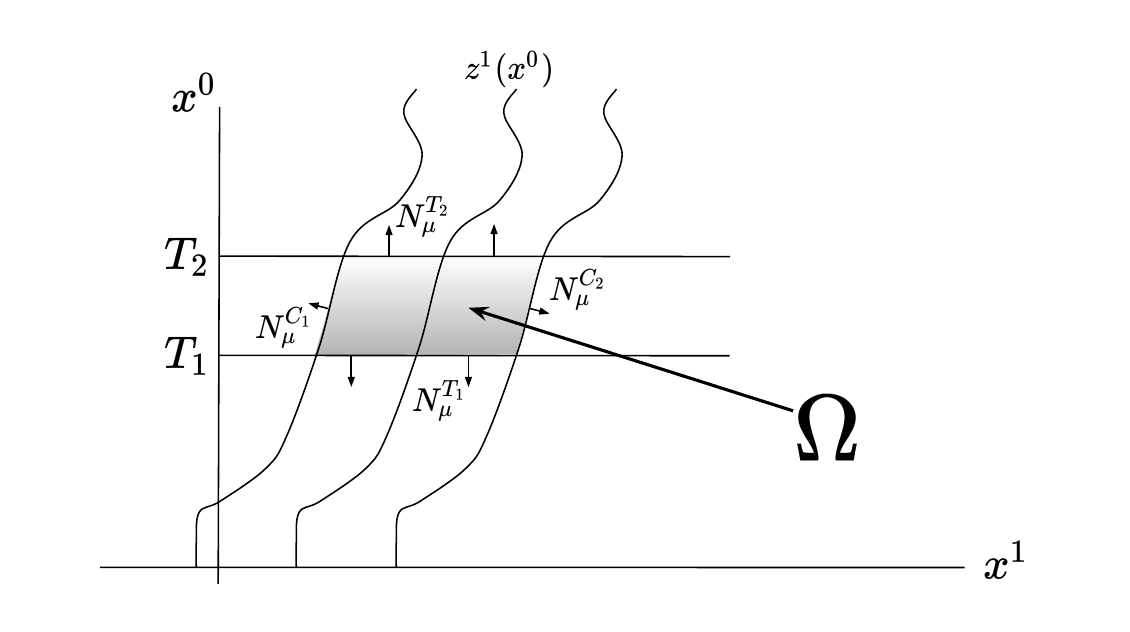}
\caption{\label{fig:region} Region of integration and its normals}
\end{figure}
and two time-like curves given by
\begin{equation}
\begin{split}
&\bar{C}_1=\{x\in \mathbb{R}^{1,1}| T_1<x^0<T_2, x^1=z^1(x^0)-\epsilon\}
\\
&\bar{C}_2=\{x\in \mathbb{R}^{1,1}| T_1<x^0<T_2, x^1=z^1(x^0)+\epsilon\}.
\end{split}
\end{equation}
The unit normal covectors for these boundaries are given by
\begin{equation}
    N^{T_1}_\mu=-\begin{pmatrix}
    1
    \\
    0
    \end{pmatrix}, \quad N^{T_2}_\mu=\begin{pmatrix}
    1
    \\
    0
    \end{pmatrix},\quad  N^{C_1}_\mu=-\begin{pmatrix}
    -\frac{\partial z^1}{\partial \tau}
    \\
    \frac{\partial z^0}{\partial \tau}
    \end{pmatrix}, \quad
     N^{C_2}_\mu=\begin{pmatrix}
    -\frac{\partial z^1}{\partial \tau}
    \\
    \frac{\partial z^0}{\partial \tau}
    \end{pmatrix}.
\end{equation}
Notice that $N_\mu^{C_2}=-N_\mu^{C_1}=n_\mu$ is the unit covector annihilated by $u^\mu$, up to a sign.
Since $T_p^{\mu \nu}$ has no support on the boundaries $\bar{C}_1$ and $\bar{C}_2$ their contribution to the surface integral vanishes, and we are left with
\begin{equation}\label{Particle Flux}
\int_{\partial \Omega} T^{\alpha \nu}_p N_\alpha dS=\int_{z^1(T_2)-\epsilon}^{z^1(T_2)+\epsilon}T^{0 \nu}_p dx^1-\int_{z^1(T_1)-\epsilon}^{z^1(T_1)+\epsilon}T^{0 \nu}_p dx^1.
\end{equation}
Recall that 
\begin{equation}
   T^{0 \nu}_p= m(x^0) \int u^0 u^\nu \delta^{(2)}(x-z(\tau))d\tau=p^\nu \delta(x^1-z^1).
\end{equation}
Plugging this into (\ref{Particle Flux}), we arrive at
\begin{equation}
    \int_{\partial \Omega} T^{\alpha \nu}_p N_\alpha dS=p^\nu(T_2)-p^\nu (T_1)=\int_{\tau_1}^{\tau_2}\frac{d p^\nu}{d \tau}d\tau.
\end{equation}
where $z^0(\tau_i)=T_i$. So our assumption of weak conservation of energy-momentum grants us the following equation
\begin{equation}
    \int_{\tau_1}^{\tau_2}\frac{d p^\nu}{d \tau}d\tau=-\int_{\partial \Omega} T_S^{\alpha \nu} N_\alpha dS,
\end{equation}
which states that the  total change in the particle's momentum is equal to the total energy flux of the field through a tube around the world-line of the particle. The L.H.S of this equation has no dependence on the width of our tubular region $\epsilon$, so it follows that the R.H.S should also have no dependence on $\epsilon$. Taking the limit as $\epsilon$ goes to zero returns
\begin{equation}\label{limit worldtube}
    \int_{\tau_1}^{\tau_2}\frac{d p^\nu}{d \tau}d\tau=-\lim_{\epsilon \to 0}\int_{\partial \Omega} T_S^{\alpha \nu} N_\alpha dS
.\end{equation}
By our assumptions two of the boundary contributions vanish:
\begin{equation}
   \lim_{\epsilon \rightarrow 0} \int_{\bar{T}_2} T^{\alpha \nu}_S N_\alpha dS=\lim_{\epsilon \rightarrow 0}\int_{z^1(T_2)-\epsilon}^{z^1(T_2)+\epsilon}T^{0\nu}_S dx^1=0,
\end{equation}
and
\begin{equation}
   \lim_{\epsilon \rightarrow 0} \int_{\bar{T}_1} T^{\alpha \nu}_S N_\alpha dS=\lim_{\epsilon \rightarrow 0}\int_{z^1(T_1)-\epsilon}^{z^1(T_1)+\epsilon}-T^{0\nu}_S dx^1=0.
\end{equation}
When we parameterize the remaining space-like boundaries by $\tau$, we obtain
\begin{equation}
    \int_{\bar{C}_2} T^{\alpha \nu}_S N^{\bar{C}_2}_\alpha dS=\int_{\tau_1}^{\tau_2}N_\alpha^{C_2}T^{\alpha \nu}_S(z^0(\tau),z^1(\tau)+\epsilon) d\tau,
\end{equation}
and
\begin{equation}
    \int_{\bar{C}_1} T^{\alpha \nu}_S N^{\bar{C}_1}_\alpha dS=-\int_{\tau_1}^{\tau_2}N_\alpha^{C_2}T^{\alpha \nu}_S (z^0(\tau),z^1(\tau)-\epsilon) d\tau_.
\end{equation}
Plugging these back into equation (\ref{limit worldtube}) returns
\begin{equation}
    \int_{\tau_1}^{\tau_2}\frac{d p^\nu}{d \tau}d\tau=-\int_{\tau_1}^{\tau_2}\bigg[n_\mu T^{\mu \nu}_S (z^0,x^1)\bigg]_{x^1=z^1} d\tau.
\end{equation}
Since this holds for all $\tau_1, \tau_2$, Equation (\ref{K Force 1}) immediately follows.
\end{proof}
\section{Well-Posedness of Joint Evolution}
In this section we present the global well-posedness result for a scalar particle perturbed by some smooth radiation. We work in the fixed Lorentz frame where the scalar charge was stationary and remained at the origin  for all time $x^0<0$. The radiation will appear as smooth initial data in the IVP for $U$.
\subsection{Joint IVP Set Up} We begin by setting up the joint initial value problem for our field-particle system. The Cauchy problem for the scalar field $U$ is given by
\begin{equation}\label{Joint 3}
\left\{\begin{array}{rcl}
\eta^{\mu \nu}\partial_\mu \partial_\nu U(x)&=& a\int \delta^{(2)}(x-z(\tau))d\tau 
\\
U(0,x^1)&=&-\frac{a}{2}|x^1|+V_0(x^1)
\\
\partial_0 U(0,x^1)&=&V_1(x^1),
\end{array}\right.
\end{equation}
where $V_0$ and $V_1$ are smooth functions which represent the external radiation. The term $-\frac{a}{2}|x^1|$ is included in the initial data to represent the field that was sourced by the stationary charge. It is compatible with the particle being at the origin for all $x^0<0$.  Since our evolution equation is linear, it is natural to split the Cauchy problem into three parts by setting 
\begin{equation}\label{U Decomposition}
U=V+U_\text{stat}+U_\text{source}
\end{equation}
where
\begin{equation}\label{Freely Evolving}
\left\{\begin{array}{rcl}
\eta^{\mu \nu}\partial_\mu \partial_\nu V&=&0
\\
V(0,x^1)&=&V_0(x^1)
\\
\partial_0 V(0,x^1)&=&V_1(x^1)
\end{array}\right. \quad \left\{\begin{array}{rcl}
\eta^{\mu \nu}\partial_\mu \partial_\nu U_{\text{stat}}&=&0
\\
U_{\text{stat}}(0,x^1)&=&-\frac{a}{2}|x^1|
\\
\partial_0 U_{\text{stat}}(0,x^1)&=&0
\end{array}\right.
\end{equation}
\begin{equation}\label{Source PDE}
\left\{\begin{array}{rcl}
\eta^{\mu \nu}\partial_\mu \partial_\nu U_\text{source}(x)&=&\int a\delta^{(2)}(x-z(\tau))d\tau
\\
U_\text{source}(0,x^1)&=&0
\\
\partial_0 U_\text{source}(0,x^1)&=&0
\end{array}\right.
\end{equation}
The solutions to (\ref{Freely Evolving}) are given by d'Alembert's formula while the solution to (\ref{Source PDE}) is given by Duhamel's principle. $U_\text{source}$ can be written as an integral equation which depends on the trajectory of the charge. From these solutions it is easy to calculate the R.H.S of our force law:
\begin{equation}\label{Force Law}
   \frac{dp^1}{d\tau}= -\bigg[n_\mu T_S^{\mu 1}(x^0,x^1)\bigg]_{x^1=z^1}=-\frac{a^2}{2}u^1+a\partial_1 V(x^0,z^1).
\end{equation} 
(See \cite{SEA} for a similar calculation.)
\begin{remark}
The force law derived from the conservation of energy is similar to the one derived from the principle of stationary action. However, the singular ``self-force" term has been determined rather than ignored, and the expression for it guarantees the conservation of the system's total energy-momentum.
\end{remark}
Written in terms of the dynamical mass and momentum, the initial value problem for the charge's trajectory is
\begin{eqnarray}\label{Joint 1}
    \frac{d z^1}{dx^0} & = & \frac{p^1}{\sqrt{m^2+(p^1)^2}}, \quad 
\\
\label{Joint 2}
   \frac{d p^1}{d x^0} & = & -\frac{a^2}{2}\frac{p^1}{\sqrt{m^2 +(p^1)^2}}+a\frac{m}{\sqrt{m^2+(p^1)^2}}\partial_1 V(x^0,z^1), 
\end{eqnarray}
with
\begin{equation}
  z^1(0)=0,  \qquad p^1(0)=0.
\end{equation}
We now state and prove the first main result of this paper.
\begin{theorem}\label{Global}
For any set of particle parameters $\{\tilde m>0,a\in \mathbb{R}\setminus\{0\}\}$, and for any set of sufficiently small, smooth, compactly supported functions $V_0(x^1)$, $V_1(x^1)$, the joint initial value problem given by (\ref{Joint 1}), (\ref{Joint 2}), and (\ref{Joint 3}) admits a unique global-in-time solution with $U$ belonging to the space of Lipschitz continuous scalar fields on $\mathbb{R}_{\geq 0}\times \mathbb{R}$ and $z$ a Lipschitz continuous world-line.
\end{theorem}
\begin{remark}
The smallness condition taken on the initial data is $||V_0||_{L^\infty} +\frac{1}{2}||V_1||_{L^1} \leq \frac{\tilde{m}}{|a|}$, and is placed to ensure that the mass $m(x^0)$ is bounded from below.
\end{remark}
\begin{remark}
    In fact, the regularity of the world-line $z$ can be improved to $C^\infty$ by applying a bootstrap argument to the system of differential equations (\ref{Joint 1}), (\ref{Joint 2}). However, the singularity appearing in (\ref{Joint 3}) clearly marks that we cannot expect anything better than Lipschitz regularity for the scalar field $U$.
\end{remark}
\subsection{Proof of Well-Posedness}

{\em Strategy for the proof}:  Imagine that instead of solving for the joint evolution of the particle and the field, we   solve for the dynamics of one when the other is given. For a given charge trajectory one can solve for $U$ by plugging the trajectory into equation (\ref{Joint 3}). Conversely, given the dynamics of the field one can solve for a test charge's trajectory via (\ref{Force Law}). Consider what happens if one were to recursively define a sequence of trajectories and field solutions by using the $i$th trajectory to solve for the $(i+1)$th field solution via (\ref{Joint 3}), and then using the $(i+1)$th field to solve for the $(i+1)$th trajectory via (\ref{Force Law}), \textit{ad infinitum}. If this sequence converges, it would converge to a trajectory which sources the same field solution that guides it. This is the key idea behind our proof, and we will see that this process does converge to a unique joint evolution. 

In this section we provide a proof of the well-posedness of the joint IVP. We will do this by transforming our joint IVP into a set of integral equations. The solution to (\ref{Source PDE}) can be written in the form of an integral equation using Duhamel's principle
\begin{equation}\label{Duhamel}
\begin{split}
    U_{\text{source}}(x^0,x^1)&=\frac{1}{2}\int_{0}^{x^0}\int_{x^1-(x^0-t)}^{x^1+(x^0-t)}a\int \delta^{(2)}(x-z(\tau))d\tau ds dt
    \\
    &=\frac{1}{2}\int_{0}^{x^0}\frac{a}{u^0(t)} \chi_{[x^1-(x^0-t),x^1+(x^0-t)]}(z^1(t))dt.
\end{split}
\end{equation}
Since our system's dynamics depend only on the field evaluated along the world-line of the particle, it suffices to consider 
\begin{equation}\label{Source Int Equation}
\begin{split}
    U_{\text{source}}(x^0,z^1(x^0))
    :=\frac{a}{2}\int_{0}^{x^0}\frac{1}{u^0(t)} dt=\frac{a}{2}\int_{0}^{x^0}\frac{m(t)}{\sqrt{m^2(t)+p^2(t)}} dt.
\end{split}
\end{equation}
\begin{theorem}\label{Global Int}
Let $U_\text{stat}$ and $V$ be the solutions to \eqref{Freely Evolving}.  Given the conditions of (\ref{Global}), there exists a unique, global in time solution to the following set of integral equations 
\begin{eqnarray}\label{Int Eq 1}
    Q(x^0)&=&\int_0^{x^0} \frac{p}{\sqrt{m^2 + p^2}}dt
\\
\label{Int Eq 2}
    p(x^0)&=&\int_0^{x^0}-\frac{a^2}{2}\frac{p}{\sqrt{m^2 +p^2}}+a\frac{m}{\sqrt{m^2+p^2}}\partial_1 V(t,Q(t))\  dt
\\
\label{Int Eq 3}
    W(x^0)&=&\int_{0}^{x^0}\frac{a}{2}\frac{m}{\sqrt{m^2+p^2}} dt,
\end{eqnarray}
where $m(t,Q(t),W(t)):=\tilde{m}-aU_\text{stat}(t,Q(t))-aV(t,Q(t))-aW(t)$.
\end{theorem}
\begin{proof}
We treat this a fixed-point problem. 
Introduce $q(\cdot)=\begin{pmatrix} Q(\cdot)
\\
p (\cdot)
\\
W (\cdot)
\end{pmatrix}$ so that we may write (\ref{Int Eq 1}), (\ref{Int Eq 2}), and (\ref{Int Eq 3}) as
\begin{equation}
    q(x^0)=\int_{0}^{x^0} f(q(t),t)dt:= F_{x^0} (q(\cdot)).
\end{equation}
We seek to prove the existence and uniqueness of a fixed point for the function $F_{*}(q(\cdot))$, which maps a given curve in $\mathbb{R}^3$ to another curve in $\mathbb{R}^3$. Notice that any such fixed point will have a bounded first derivative in all three components. Define $L_{k,\gamma}(\mathbb{R}_{\geq 0})$ to be the normed set of Lipschitz  continuous functions $l(\cdot):\mathbb{R}_{\geq 0} \rightarrow \mathbb{R}$ with Lipschitz constant $k$, and which satisfy $l(0)=0$. We equip this space with the metric induced by the weighted $L^\infty$ norm
\begin{equation}
    ||l(\cdot)||_\gamma=\sup_{t\geq 0}e^{-\gamma t}|l(t)|<\infty.
\end{equation}
We extend this definition for maps from $\mathbb{R}_{\geq 0}\rightarrow \mathbb{R}^n$ by defining 
\begin{equation}
    L_{\vec{k},\gamma}(\mathbb{R}_{\geq0})=\prod_{i=1}^n L_{k_i, \gamma}(\mathbb{R}_{\geq0})
\end{equation}
(Cartesian product), where $\vec{k}=(k_1,k_2, ... k_n)$. 
\begin{remark} For each fixed $\vec{k}$, $L_{\vec{k},\gamma}$ is a one-parameter family of metric spaces that share the same elements but differ in their equipped metric.  
\end{remark}
\begin{lemma} Fixing both $\vec{k}$ and $\gamma>0$, we have that 
$L_{\vec{k},\gamma}$ equipped with the metric induced by $||\cdot||_\gamma$ is a complete metric space.
\end{lemma}
\begin{proof}
 Corollary (\ref{Complete Subset cor}) in the appendix.
\end{proof}
Letting $K$ be an upper bound for $\frac{a^2}{2} +||a\partial_1 V_0||_{L^\infty}+||aV_1||_{L^\infty}$, and $\vec{k}=(1,K,\frac{|a|}{2})$, it becomes clear that our desired fixed point should, if it exists, reside in $L_{\vec{k},\gamma}$.

\begin{proposition}\label{contraction}
$F_*(q(\cdot)):L_{\vec{k},\gamma}\rightarrow L_{\vec{k},\gamma}$ is a well defined mapping.
\end{proposition}
\begin{proof}
Let $q=\begin{pmatrix}Q(\cdot) \\ p(\cdot) \\ W(\cdot)
\end{pmatrix}\in L_{\vec{k},\gamma}$, and write $F_*(x(\cdot))=\begin{pmatrix}P(*) \\ \rho(*) \\ S(*)
\end{pmatrix}$. Notice
\begin{equation}
\begin{split}
    \forall \theta, \theta' \geq 0, \quad |\rho(\theta)-\rho(\theta')|&=\left|\int_\theta^{\theta'}-\frac{a^2}{2}\frac{p}{\sqrt{m^2 +p^2}}+a\frac{m}{\sqrt{m^2+p^2}}\partial_1 V(t,Q(t))\ dt\right|
    \\
    &\leq K|\theta-\theta'|.
\end{split}
\end{equation}
So $\rho(*)\in L_{K,\gamma}(\mathbb{R}_{\geq 0})$. The proofs for $P(*)$ and $S(*)$ are analogous.
\end{proof}
\par
In order to prove the existence of a unique fixed point of $F_*(q(\cdot))$, we need to show it is a contraction mapping on $L_{\vec{k},\gamma}$. Towards this, we prove two lemmas.
\begin{lemma}\label{Scaled mass bound}
For $|Q(t)|\leq t$ and $|W(t)|\leq \frac{|a|}{2}t$, the quantity $m_V:=\tilde m-|a|(||V_0||_{L^\infty}+\frac{1}{2}||V_1||_{L^1})$ is a lower bound for $m(Q(t),W(t))$, and is positive by our smallness assumption on $V_0$, $V_1$.
\end{lemma}
\begin{proof} (\ref{Scaled mass bound}). Since $|Q(t)|\leq t$, (\ref{Freely Evolving}) lets us compute that
\begin{equation}
    U_\text{stat}(t,Q(t))=-\frac{a}{2}t.
\end{equation}
Thus,
\begin{equation}
\begin{split}
    m(W(t),Q(t))&=\tilde m-aV(Q(t))+\frac{a^2}{2}t -aW(t)
    \\
    &\geq\tilde m-aV(Q(t))\geq \tilde{m}-|a| ||V||_{L^\infty}\geq m_V.
\end{split}
\end{equation}
\end{proof}
\begin{lemma}\label{Contraction Mapping}
  Let $Y_\gamma$ be a complete metric space equipped with the metric induced by $||\cdot||_\gamma$,  such that $F_{*}(q(\cdot)):Y_\gamma \rightarrow Y_\gamma $ is a well defined mapping. Suppose there exists an $L<\gamma$ such that   for all $q_1(\cdot),q_2(\cdot) \in Y_\gamma$
\begin{equation}
    ||f(q_2(\tau),\tau)-f(q_1(\tau),\tau)||_X \leq L ||q_2(\tau)-q_1(\tau)||_X.
\end{equation}
  Then $F_{*}(q(\cdot)):Y_\gamma \rightarrow Y_\gamma$ is a contraction mapping.
 \end{lemma}
\begin{proof}
Theorem (\ref{Unique Global}) in the appendix. \end{proof}
\begin{proposition}\label{Lipchitz}
There exists an $L<\infty$ independent of $\gamma$ such that for all $q_1(\cdot), q_2(\cdot) \in L_{\vec{k},\gamma}$, 
\begin{equation}
    ||f(q_2(t),t)-f(q_1(t),t)||\leq L||q_2(t)-q_1(t)||.
\end{equation}
\end{proposition}
\begin{proof} %(\ref{Lipchitz}). 
Fix $t\geq 0$. We omit arguments of $t$ for ease of notation, and write $q_1=\begin{pmatrix}Q \\ p \\ W
\end{pmatrix},  q_2=\begin{pmatrix}P \\ \rho \\ S
\end{pmatrix}$. By the triangle inequality we have
\begin{equation}
\begin{split}
    &||f(q_2)-f(q_1)||\leq
    \\
&\leq \sum_{i=1}^3 |f_i\begin{pmatrix}Q \\ p \\ W
\end{pmatrix}-f_i\begin{pmatrix}P \\ p \\ W
\end{pmatrix}|+|f_i\begin{pmatrix}P \\ p \\ W
\end{pmatrix}-f_i\begin{pmatrix}P \\ \rho \\ W
\end{pmatrix}|+|f_i\begin{pmatrix}P \\ \rho \\ W
\end{pmatrix}-f_i\begin{pmatrix}P \\ \rho \\ S
\end{pmatrix}|,
\end{split}
\end{equation}
where $f_i$ denotes the components of $f$. There are nine terms that we wish to bound. For the sake of brevity, we will restrict ourselves to the least trivial bound. By the mean value theorem there exists an $O$ between $Q$ and $P$ such that
\begin{equation}
    |f_2\begin{pmatrix}Q \\ p \\ W
\end{pmatrix}-f_2\begin{pmatrix}P \\ p \\ W
\end{pmatrix}|=|Q-P||D_1f_2\begin{pmatrix}O \\ p \\ W
\end{pmatrix}|.
\end{equation}
We wish to bound $D_1 f_2$. Recall
\begin{equation}
     f_2\begin{pmatrix}O \\ p \\ W
\end{pmatrix}=-\frac{a^2}{2}\frac{p}{\sqrt{m^2 +p^2}}+a\frac{m}{\sqrt{m^2+p^2}}\partial_1 V(O).
\end{equation}
Since $O$ is between $Q$ and $P$, $m(O,W)$ is bounded from below by $m_V$. Also, $|\frac{dm(Q,W)}{dQ}(O,W)|=|a\partial_1V(O)|$ which is bounded from above by $K$. Let $K'$ be the Lipschitz constant of $\partial_1 V$. Tedious calculations then yield the following bound
\begin{equation}
    |D_1f_2\begin{pmatrix}O \\ p \\ W
\end{pmatrix}|\leq \frac{a^2}{2}\frac{K}{m_V} +\frac{K^2}{m_V}+|a|K'=:M_{1,2}.
\end{equation}
Performing a similar analysis on the other 8 terms yields a bound of the form
\begin{equation}
\begin{split}
    ||f(q_2)-f(q_1)||&\leq |Q-P|\sum_{j=1}^3 M_{1,j}+|p-\rho|\sum_{j=1}^3 M_{2,j}+|W-S|\sum_{j=1}^3 M_{3,j}
    \\
    &\leq ||q_2-q_1||\sqrt{(\sum_{j=1}^3 M_{1,j})^2+(\sum_{j=1}^3 M_{2,j})^2+(\sum_{j=1}^3 M_{3,j})^2}.
\end{split}
\end{equation}
\end{proof}
\begin{corollary}
Applying Lemma (\ref{Contraction Mapping}) with $Y_\gamma= L_{\vec{k},\gamma}$ and $\gamma>L$, we find that there exists a unique $q\in L_{\vec{k},\gamma}(\mathbb{R}_{\geq 0})$  satisfying 
\begin{equation}
    q(x^0)=F_{x^0} (q(\cdot))=\int_{0}^{x^0} f(q(t),t)dt
\end{equation}
for all $x^0\geq0$. This concludes our proof of Theorem (\ref{Global Int}).
\end{corollary}
\end{proof}
\begin{porism}
Letting $q(x^0)=\begin{pmatrix}Q(x^0) \\ p(x^0) \\ W(x^0)
\end{pmatrix}$, and defining the function $W_\text{source}(x^0,x^1)$ via
\begin{equation}
W_\text{source}(x^0,x^1)=\int_{0}^{x^0}\frac{a}{2}\sqrt{1-(\frac{dQ}{dx^0})^2} \chi_{[x^1-(x^0-t),x^1+(x^0-t)]}(Q(t))dt,
\end{equation}
we find that $Q(x^0)$, $p(x^0)$, and $U(x^0,x^1):=V(x^0,x^1)+U_\text{stat}(x^0,x^1)+W_\text{source}(x^0,x^1)$ are unique global solutions to the joint IVP given by equations (\ref{Joint 1}), (\ref{Joint 2}), and (\ref{Joint 3}). This concludes our proof of Theorem (\ref{Global}).
\end{porism}
\section{Asymptotic Stability of Scalar Particle}
The force law that describes the charge's dynamics given by (\ref{Force Law}) is very similar to the one given by the principle of stationary action (\ref{Lorentz force}), with the exception of a well-defined self-force. We derived this from the assumption that the total energy-momentum of our system is conserved, so it is not unexpected that the self-force is restorative, in magnitude proportional to the particle's velocity but with the opposite sign.  \href{https://sites.math.rutgers.edu/~laf230/Motion_of_Perturbed_Scalar_Particle_V2.gif}{This animation}\footnote{Link to animation: \href{https://sites.math.rutgers.edu/~laf230/Motion_of_Perturbed_Scalar_Particle_V2.gif}{https://sites.math.rutgers.edu/\~laf230/Motion\_of\_Perturbed\_Scalar\_Particle\_V2.gif}}
  of a scalar particle's sourced field showcases how as the particle's motion is disturbed so too is the field that it is sourcing. These disturbances in the sourced field carry energy, so to preserve the total energy of the system, the field disturbance's energy must have been sourced by the particle's kinetic energy. This is interpreted as an effective ``self-force" which was conjectured to sap away the particle's kinetic energy until all motion ceases.

We will show that, as expected, in the case of a scalar particle perturbed by some incoming radiation, the self-force causes the charge to asymptotically return towards rest.
\subsection{Asymptotic behavior of self-force}
Take $V_0$, $V_1$ sufficiently small, smooth, and compactly supported. 
Although the form of (\ref{Force Law}) was beneficial for proving well-posedness, we will find it best to convert the equation to one for $\frac{du^1}{dx^0}$ to better study the charge's motion.   Recalling that the dynamical mass of the charge is given by equation (\ref{dyn mass}), we obtain
\begin{equation}\label{product rule}
    m\frac{du^1}{d\tau}=au^1 \frac{d U(z(\tau))}{d \tau}-\frac{a^2}{2}u^1 + \partial_1 V.
\end{equation}
It is easy to verify from (\ref{Duhamel}) that $$\frac{dU_\text{source}(z(\tau))}{d\tau}=\frac{a}{2},$$ which implies via (\ref{U Decomposition}) a distinct cancellation in (\ref{product rule}) between the radiation-reaction self-force and the mass loss generated by $U_\text{source}$. 
\begin{remark}
    We note that a similar cancellation holds for point charges coupled to {\em massive} scalar fields in one space dimension \cite{FETZ}.
\end{remark}
Despite this cancellation, our final equation of motion still contains a restoring term from the mass contribution of the freely evolving stationary state solution $U_\text{stat}$.   We are left with
\begin{equation}\label{acceleration law}
    m \frac{du^1}{dx^0}=-\frac{a^2}{2}u^1+au^1 \frac{dV(x^0,z^1(x^0))}{dx^0}+\frac{a}{u^0}\partial_1V(x^0,z^1).
\end{equation}
The terms in (\ref{acceleration law}) can be split up into two categories, those which depend on the external radiation and those which do not. We write
\begin{equation}
    m\frac{d u^1}{dx^0}=F_\text{self}+ F_\text{ext},
\end{equation}
where 
\begin{equation}
F_{\text{self}} =-\frac{a^2}{2}u^1,
\end{equation}
and
\begin{equation}
F_{\text{ext}} =au^1\frac{d V(x^0,z^1(x^0))}{dx^0}+\frac{a}{u^0}\partial_1 V.
\end{equation}
Since the radiation was compactly supported and is propagating at the speed of light, we expect it to perturb the motion of the particle and then propagate away. Eventually we expect only the self-forces to remain, and it is clear that the particle will asymptotically tend towards rest as long as the dynamical mass does not grow too quickly. 

We begin by proving that the external radiation does eventually propagate away from the particle. 
\begin{lemma}\label{Displacement} Let $z^\mu$, $p^\mu$, and $U$ satisfy the joint IVP given by (\ref{Joint 1}), (\ref{Joint 2}), and (\ref{Joint 3}) with the same conditions as (\ref{Global}). Let $v=\frac{dz^1}{dx^0}$. Then
\begin{equation}
    \int_{0}^\infty 1-|v(t)| \ dt =\infty.
\end{equation}
\end{lemma}
\begin{proof}
Suppose not. Then
\begin{equation}
    \int_{0}^\infty 1-|v(t)|\  dt= D <\infty.
\end{equation}
Recall that 
\begin{equation}
    v=\frac{p^1}{\sqrt{m^2 +(p^1)^2}}.
\end{equation}
We proceed by proving estimates about $p^1$ and $m$ which will contradict the rate of growth of $|v|$. The first will be the rough linear estimate for $p^1$. From (\ref{Force Law}) we have
\begin{equation}
\frac{d p^1}{d x^0}=-\frac{a^2}{2}v +a\sqrt{1-v^2}\partial_1 V.
\end{equation}
So $|p^1(x^0)|\leq Rx^0$ where $R\geq\frac{a^2}{2}+||a\partial_1 V_0||_{L^\infty}+||aV_1||_{L^\infty}$. To estimate $m(t)$, recall 
\begin{equation}
    m(x^0)= \tilde m-aV(x^0,z^1(x^0))+\frac{a^2}{2}x^0 -aU_\text{source}(x^0,z^1(x^0)).
\end{equation}
We bounded $m$ from below in the previous section by showing that $|aU_\text{source}(x^0,z^1(x^0))|\leq \frac{a^2}{2}x^0$. Using our contradiction hypothesis, we will now prove a stronger bound. For $x^0 \geq 0$ we have
\begin{equation}
\begin{split}
|aU_\text{source}(x^0,z^1(x^0))|&= \frac{a^2}{2}\int_0^{x^0}\sqrt{1-v(t)^2} dt
\\
&\leq \frac{a^2}{2}\bigg(\sqrt{\int_0^{x^0} 1-|v(t)| dt}\bigg)\cdot\bigg(\sqrt{\int_0^{x^0} 1+|v(t)| dt}\bigg) 
\\
&\leq \frac{a^2}{2} \sqrt{D}\sqrt{2x^0}.
\end{split}
\end{equation}
By the conditions on the radiation we see that the linear term will dominate $m(x^0)$. So, there exists a time $T>0$, and a slope $M>0$ such that for all $x^0>T$ we have that $m(x^0)\geq Mx^0$. Using our estimates for $p^1$ and $m$, we can conclude with a final estimate for $v$. For all $t>T$
\begin{equation}
    |v(t)|=\frac{|p^1|}{\sqrt{m^2 +(p^1)^2}}\leq \frac{Rt}{\sqrt{(Mt)^2 +(Rt)^2}}=\frac{R}{\sqrt{M^2+R^2}}<1.
\end{equation}
But this implies that $1-|v(t)|$ is bounded from below, so
\begin{equation}
    \int_{0}^\infty 1-|v(t)|\ dt =\infty.
\end{equation}
as desired. 
\end{proof}
\begin{corollary}\label{cor:Vanish}
Since $V_0$ and $V_1$ are compactly supported, it follows that all external force terms in $F_\text{ext}$ will vanish in finite time.
\end{corollary}
\begin{proof}
    It suffices to show that $\partial_1V$, $\partial_0V$ evaluated along the particle's world-line vanish after finite time. By d'Alambert's formula we have
    \begin{equation}\label{Vanish}
    \begin{split}
    \partial_1V(x^0,z^1(x^0))=\frac{1}{2}[V_0(z^1(x^0)+x^0)+V_0(z^1(x^0)-x^0)+V_1(z^1(x^0)+x^0)-V_1(z^1(x^0)-x^0)].
    \\
    \partial_0V(x^0,z^1(x^0))=\frac{1}{2}[V_0(z^1(x^0)+x^0)-V_0(z^1(x^0)-x^0)+V_1(z^1(x^0)+x^0)+V_1(z^1(x^0)-x^0)].
    \end{split}
    \end{equation}
    Lemma (\ref{Displacement}) implies that the quantities $z^1(x^0)\pm x^0$ grow arbitrarily large, so each term in (\ref{Vanish}) will vanish once $z^1(x^0)\pm x^0$ leaves the support of $V_0, V_1$.
\end{proof}
We conclude this section by stating and proving the second main result of this paper: that the perturbed charge asymptotically returns towards rest.
\begin{theorem}\label{Rest}
Let $z$, $U$ satisfy the joint IVP given by (\ref{Joint 1}) (\ref{Joint 2}), and (\ref{Joint 3}) for a single scalar particle with positive bare mass and non-zero real scalar charge, where $V_0$ and $V_1$ satisfy the same conditions as in Theorem (\ref{Global}).
Then for every $\epsilon>0$, there exists a $T_\epsilon>0$ such that $|u^1(x^0)|<\epsilon$ for all $x^0>T_\epsilon$.
\end{theorem}
\begin{proof} By Corollary~\ref{cor:Vanish}, there exists a time $T$ such that for all $x^0>T$, $u^1(x^0)$ satisfies
\begin{equation}\label{Self ODE}
    \frac{du^1}{dx^0}=-\frac{a^2u^1}{2m}.
\end{equation}
If at some time $T'>T$ we have that $u^1=0$, then it will remain 0 for all time afterwards and the theorem holds. Suppose $|u^1|>0$ for all time after $T$. Dividing both sides of (\ref{Self ODE}) by $u^1$ and integrating returns
\begin{equation}
    \ln(|u^1(x^0)|)=\ln(|u^1(T)|)-\frac{a^2}{2}\int_T^{x^0}\frac{1}{m}dt.
\end{equation}
Recall that $m$ is strictly positive and increases at most linearly, thus there exists an $A$ such that for all $x^0>T>0$, $m<Ax^0$. It follows that 
\begin{equation}
    \ln(|\frac{u^1(x^0)}{u^1(T)}|)<-\frac{a^2}{2A}(\ln(\frac{x^0}{T})).
\end{equation}
So 
\begin{equation}
    |u^1(x^0)|< |u^1(T)|\big(\frac{x^0}{T}\big)^{-\frac{a^2}{2A}}.
\end{equation}
Since $\frac{a^2}{2A}>0$, we see that $u^1$ decays to zero asymptotically.
\end{proof}
\section{Summary and Outlook}
In this paper we studied the joint evolution of a point particle coupled to a scalar field in $1+1$ dimensions. A rigorous derivation for the force on a charged particle was given from the assumption of conservation of energy-momentum, and from this we calculated the self-force which result from back-reaction. The scalar self-force that charged particles experience is restorative, proportional to negative velocity, and causes the charge to asymptotically return towards rest after being perturbed by radiation. Written explicitly:
\begin{equation}
    m\frac{d u^1}{dx^0}=F_\text{self}+ F_\text{ext},
\end{equation}
where 
\begin{equation}
F_{\text{self}} =-\frac{a^2}{2}u^1,
\end{equation}
and $a$ denotes the scalar charge.
We also proved the well-posedness for the joint evolution problem.
\par 
Although our study was meant to act as a toy model, it has produced some fruitful results. They indicate that for universes in which fields coupled to point particles evolve with a certain level of regularity, the joint evolution problem of fields and particles is well-posed under the condition that their dynamics preserve the total energy-momentum. The Maxwell-Lorentz theory of point charges in $3+1$ dimensions does not give rise to the level of regularity sufficient for such joint evolutions to exist. BLTP, a higher-order modification to Maxwell-Lorentz theory has found success in $3+1$ dimensions because of the regularity of its field's evolution. However, the complexity of the BLTP electromagnetic self-force in $3+1$ dimensions does not admit an easy way to   investigate whether or not   the process of back-reaction   stabilizes   a perturbed particle   in that case  (see \cite{CarKie24} for an interesting preliminary result in this direction.) Working in $1+1$ dimensions has provided us with a much simpler set of dynamics, and we were able to prove   asymptotic stability   in the case of a scalar particle. The mathematical simplicity afforded to us in this lower-dimensional setting makes our results approachable and easily understood,  while we hope   many of the physically intuitive arguments still carry over to their higher-dimensional counterparts.

There are multiple avenues that we will take to further investigate the process of back-reaction,  still within this simplified framework:  In ongoing work 
 \cite{FETZ}   we generalize our results to the case of {\em massive} scalar fields, which we believe will provide researchers with a clearer picture of back-reaction in higher-order modifications of $3+1$ dimensional field theories.    We would like to study a scattering problem for this model, one analogous to \cite{Hoang}, in which the initial data is prescribed at the infinite past. Another possibility is to study the general initial value problem with arbitrary data for the field that is compatible with a point charge (although in that case one might expect additional singularities to form in the field on the light cone of the initial data position, as was shown in \cite{DeHa} for Maxwell-Lorentz electrodynamics.)   We also plan to pursue a project investigating {\em gravitational} back-reaction, to see whether a gravitational self-force arises when deriving the force law from the principle of energy-momentum conservation.
\section{Acknowledgment} We are grateful the anonymous referees for their sharp eyes and insightful remarks.  We thank Aditya Agashe and Ethan Lee for their assistance in drawing Figure \ref{fig:region}.
\begin{appendix}
\section{Well-Posedness of Integral Equations}
In this section we will provide brief proofs of well-known results regarding the well-posedness of integral equations. We will do this with the help of Banach fixed point theorem.

We first recall a standard result:
\begin{lemma}Let $X$ be a metric space, and let $C_{\gamma}^0(\mathbb{R}_{\geq 0},X)$ be the normed set of continuous maps from $\mathbb{R}_{\geq 0}\rightarrow X$ such that for all $x(\cdot) \in C_{\gamma}^0(\mathbb{R}_{\geq 0},X)$ we have
\begin{equation}
    ||x(\cdot)||_\gamma=\sup_{t\geq 0}e^{-\gamma t}||x(t)||_X<\infty.
\end{equation}
Then $C_{\gamma}^0(\mathbb{R}_{\geq 0},X)$ is a Banach space.
\end{lemma}

Let $f:X\times \mathbb{R}_{\geq 0} \rightarrow X$. Define 
\begin{equation}
     F_t (q(\cdot)|x_0):=x_0+\int_{0}^t f(q(\tau),\tau)d\tau.
\end{equation}
\begin{theorem}\label{Unique Global}
  Let $Y_\gamma$ be a closed subset of $C_{\gamma}^0(\mathbb{R}_{\geq 0},X)$ such that $F_{*}(x(\cdot)|x_0):Y_\gamma \rightarrow Y_\gamma $ is a well defined mapping. Suppose there exists an $L<\gamma$ such that   for all $q_1(\cdot),q_2(\cdot) \in Y_\gamma$
\begin{equation}
    ||f(q_2(\tau),\tau)-f(q_1(\tau),\tau)||_X \leq L ||q_2(\tau)-q_1(\tau)||_X.
\end{equation}
  Then $F_{*}(x(\cdot)|x_0)$ restricted to $Y_\gamma$ is a contraction mapping on a complete metric space. Thus, $Y_\gamma$ contains a unique fixed point $q(\cdot)$ which satisfies
\begin{equation}
    q(t)=F_{t}(q(\cdot)|x_0)=x_0 +\int_0^t f(q(\tau),\tau)d\tau.
\end{equation}
for all $t\geq 0$. 
\end{theorem}
\begin{proof}
We wish to show that for all $x_1(\cdot),x_2(\cdot) \in Y_\gamma$,
\begin{equation}
    ||F_{*}(x_2(\cdot)|x_0)-F_{*}(x_1(\cdot)|x_0)||_{\gamma}\leq \frac{L}{\gamma} ||x_2(\cdot)-x_1(\cdot)||_\gamma.
\end{equation}
By definition
\begin{equation}
\begin{split}
    ||F_{*}(x_2(\cdot)|x_0)-F_{*}(x_1(\cdot)|x_0)||_{\gamma}&=\sup_{t\geq 0}e^{-\gamma t}||\int_{0}^t f(x_2(\tau),\tau)-f(x_1(\tau),\tau)d\tau||_X
    \\
    &\leq \sup_{t\geq 0}e^{-\gamma t}\int_{0}^t ||f(x_2(\tau),\tau)-f(x_1(\tau),\tau)||_X d\tau
    \\
    &\leq \sup_{t\geq 0}e^{-\gamma t}\int_{0}^t L ||x_2(\tau)-x_1(\tau)||_Xd\tau
    \\
    &=\sup_{t\geq 0}\int_{0}^t e^{-\gamma(t-\tau)}e^{-\gamma \tau}L||x_2(\tau)-x_1(\tau)||_Xd\tau
    \\
    &\leq \sup_{t\geq 0}\int_{0}^t e^{-\gamma(t-\tau)}L||x_2(\cdot)-x_1(\cdot)||_\gamma d\tau
    \\
    &\leq \frac{L}{\gamma}||x_2(\cdot)-x_1(\cdot)||_{\gamma}.
\end{split}
\end{equation}
\end{proof}
\par
In cases where $f:\mathbb{R}\times \mathbb{R}_{\geq 0}\rightarrow \mathbb{R}$ is bounded it is useful to consider $L_{k,\gamma}(\mathbb{R}_{\geq 0})\subset C_{\gamma}^0(\mathbb{R}_{\geq 0},\mathbb{R})$, the normed set of all Lipschitz continuous functions $l(\cdot)$ with Lipschitz constant $k$ and $l(0)=0$.
\begin{theorem}\label{Complete subset}For all $\gamma>0$,
$L_{k,\gamma}(\mathbb{R}_{\geq 0})$ is a closed subset of $C_{\gamma}^0(\mathbb{R}_{\geq 0},\mathbb{R})$, and therefore a complete metric space.
\end{theorem}
\begin{proof} Let $\{l_n(\cdot)\}$ be a Cauchy sequence in $L_{k,\gamma}(\mathbb{R}_{\geq 0})$. This has a limit $l(\cdot)$ in $C_{\gamma}^0(\mathbb{R}_{\geq 0},\mathbb{R})$. Also, we have that for every $t, t' \geq 0$ 
\begin{equation}
       k|t-t'| -|l_n(t)-l_n(t')|\geq 0.
\end{equation}
The convergence is pointwise, so taking the limit of both sides as $n$ goes to infinity yields
\begin{equation}
    k|t-t'| -|l(t)-l(t')|\geq 0.
\end{equation}
Taking the infimum over $t,t'$ yields that $l(\cdot)\in L_{k,\gamma}(\mathbb{R}_{\geq 0})$ as desired.
\end{proof}
\begin{corollary}\label{Complete Subset cor}For all $\gamma>0$, $\vec{k}\in \mathbb{R}^n$, it follows that $L_{\vec{k},\gamma}(\mathbb{R}_{\geq 0}):=\prod_{i=1}^n L_{k_i,\gamma}(\mathbb{R}_{\geq 0})$ is a closed subset of $C_\gamma^0(\mathbb{R}_{\geq 0},\mathbb{R}^n)$, and therefore a complete metric space.
\end{corollary}
\end{appendix}
%\printbibliography
\bibliographystyle{plain}

\end{document}